\documentclass{birkjour}

\title[Norm-constrained determinantal representations of polynomials]
{Norm-constrained determinantal \\ representations of multivariable polynomials}
\author[Grinshpan]{Anatolii Grinshpan}
\author[Kaliuzhnyi-Verbovetskyi]{Dmitry~S.~Kaliuzhnyi-Verbovetskyi}
\author[Woerdeman]{Hugo J.~Woerdeman}
\address{\\ Department of Mathematics \\
Drexel University\\
3141 Chestnut St.\\
Philadelphia, PA, 19104}
\email{\{tolya,dmitryk,hugo\}@math.drexel.edu}

\thanks{The authors were partially
supported by NSF grant DMS-0901628.}


\theoremstyle{plain}

\newtheorem{thm}{Theorem}[section]
\newtheorem{cor}[thm]{Corollary}
\newtheorem{lem}[thm]{Lemma}
\newtheorem{prop}[thm]{Proposition}

\numberwithin{equation}{section}

\newcommand{\ds}{\displaystyle}
\newcommand{\bs}{\bigskip}
\newcommand{\mc}{\mathcal}

\newcommand{\s}{\sigma}

\newcommand{\Hspace}[1]{\ensuremath{\mathcal{#1}}}
\theoremstyle{remark}
\newtheorem{rem}[thm]{Remark}

\newtheorem{ex}[thm]{Example}

\newcommand{\diag}{\operatorname{diag}}
\newcommand{\tdeg}{\operatorname{tdeg}}

\newcommand{\row}{\operatorname{row}}
\newcommand{\col}{\operatorname{col}}
\newcommand{\per}{\operatorname{per}}
\newcommand{\card}{\operatorname{card}}

\numberwithin{equation}{section}

\renewcommand{\bs}{\left(\begin{smallmatrix}}
\newcommand{\es}{\end{smallmatrix}\right)}
\newcommand{\bd}{\left|\begin{array}{ccc}}
\newcommand{\ed}{\end{array}\right|}

\renewcommand{\ds}{\displaystyle}
\renewcommand{\bs}{\bigskip}
\renewcommand{\mc}{\mathcal}

\renewcommand{\s}{\sigma}

\begin{document}

\maketitle

\begin{abstract} For every multivariable polynomial $p$, with $p(0)=1$, we construct a determinantal representation
$$p=\det (I - K Z ),$$
where $Z$ is a diagonal matrix with coordinate variables on the
diagonal and $K$ is a complex square matrix. Such a representation
is equivalent to the existence
 of $K$ whose principal minors satisfy certain linear relations. When norm constraints on
 $K$ are imposed, we give connections to the multivariable von Neumann inequality, Agler denominators,
 and stability. We show that if a multivariable polynomial $q$, $q(0)=0,$ satisfies the von Neumann
  inequality, then $1-q$ admits a determinantal representation with $K$ a contraction. On the other
  hand, every determinantal representation with a contractive $K$ gives rise to a rational inner
  function in the Schur--Agler class.
\end{abstract}

\section{Introduction}\label{sec:Intro}

Our object of study is determinantal representations
\begin{equation}\label{eq:repr} p(z)= \det (I_{|n|} - K Z_n ), \end{equation}
for a $d$-variable polynomial $p(z)$, $z=(z_1,\ldots , z_d)$, with
$p(0)=1$. Here $n=(n_1, \ldots , n_d)$ is in the set $\mathbb
N_0^d$ of $d$-tuples of nonnegative integers, $|n|=n_1+\cdots +
n_d$, $Z_n = \bigoplus_{i=1}^d z_i I_{n_i}$, and $K$ is a complex
square matrix. It is of interest of how and to what extent, the
algebraic and operator-theoretic properties of the polynomial
correspond to the size and norm of the matrix $K$ of its
representation.

Various determinantal representations of polynomials have been studied,
 often for polynomials over the reals: see a recent overview article \cite{Vinnikov},
  together with bibliography, and also \cite{Net-Thom}. The particular form of \eqref{eq:repr}
  has appeared before, for instance, in \cite{Borcea}. An important early result on two-variable
  polynomials was obtained by A. Kummert \cite[Theorem 1]{Kummert}.

Given a $d$-variable polynomial $p(z)$, $p(0)=1$, we consider the question of whether it can be represented in the form \eqref{eq:repr}
for some $n\in {\mathbb N}_0^d$ and some $|n|\times |n|$ complex matrix $K$, possibly subject to a constraint.
It will be shown that the unconstrained version of this question can always be answered in the affirmative (Section \ref{sec:affine-det-repr}).
The problem of minimizing the operator norm of $K$ over all representations \eqref{eq:repr} of $p$ will be seen to be more involved (Section \ref{sec:constr-det-repr}).

We will say that the \emph{multi-degree} $\deg p$ of a polynomial $p$ is $m=(m_1, \ldots , m_d)$ if $m_i=\deg_ip$ is the degree of $p$ as a polynomial of $z_i$, $i=1,\ldots,d$. The \emph{total degree} $\tdeg p$ of $p$ is the largest $|k|$ over all monomials $z^k=z_1^{k_1}\cdot\ldots\cdot z_d^{k_d}$ of $p$. For $m, n\in\mathbb N_0^d$, the inequality $m\le n$ will be meant in the usual component-wise sense: $m_i\le n_i$, $i=1,\ldots, d$.

For a matrix $K$, the principal submatrix determined by an index set $\alpha$ will be denoted by $K[\alpha]$. Given a collection of complex numbers $c_\alpha$, indexed by nonempty subsets $\alpha$ of $\{ 1, \ldots , d \}$, the Principal Minor Assignment Problem (see, e.g., \cite{Stouf,HoSch}) consists of finding a $d\times d$ matrix $K$ such that $\det K[\alpha]=c_\alpha$ for all $\alpha$.
This problem is, in general, overdetermined since the number of independent principal minors grows exponentially with the matrix size, $d$, while the number of free parameters, the matrix entries, is $d^2$. It becomes well-posed under additional assumptions on $K$ or $d$.
For theoretical and computational advances, see \cite{GrTs,LinSt,HoSt}.

A polynomial of multi-degree $(1, \ldots, 1)$, is said to be \emph{multi-affine}. For such a polynomial, the problem of finding a representation \eqref{eq:repr} with $n=(1,\ldots,1)$ is equivalent to the Principal Minor Assignment Problem. This follows by comparing the expansion
$$\det (I_d - K Z_{(1,\ldots,1)}) = 1 + \sum_{\alpha\neq\emptyset}(-1)^{\card\alpha} \det K[\alpha ] \prod_{i\in \alpha} z_i,$$
to the general form
$$p(z)=1+\sum_{\alpha\neq\emptyset} (-1)^{\card\alpha}c_\alpha \prod_{i\in \alpha} z_i$$
of a $d$-variable multi-affine polynomial $p$, $p(0)=1$.

For a general polynomial $p$, finding a determinantal
representation \eqref{eq:repr} with $n=\deg p$ may not be possible
by the same dimension count as above. It is clear that
\eqref{eq:repr} implies that $n\ge\deg p$. If $n$ is prescribed,
one may view \eqref{eq:repr} as the Principal Minor Relation
Problem formulated in Section \ref{sec:affine-det-repr}.

This paper is largely motivated by our study \cite{GKVVW} of the multivariable von Neumann inequality
and the discrepancy between the Schur and Schur--Agler norms of
analytic functions on the unit polydisk $${\mathbb D}^d= \{
z=(z_1,\ldots,z_d) \in {\mathbb C}^d\colon |z_i|<1,\ i=1,\ldots,d
\}.$$

 The \emph{Schur class}, consisting of analytic $\mathcal{L}(\Hspace{U},\Hspace{Y})$-valued
functions $f$ on
${\mathbb D}^d$ such that
\begin{equation}\label{eq:sup-norm}
\|f\|_\infty:=\sup_{z \in {\mathbb D}^d} \| f(z) \|\le 1,
 \end{equation}
will be denoted by $\mathcal{S}_d(\Hspace{U},\Hspace{Y})$. Here
$\mathcal{L}(\Hspace{U},\Hspace{Y})$ is the Banach space of
bounded linear operators from a Hilbert space $\Hspace{U}$ to a
Hilbert space $\Hspace{Y}$. The \emph{Schur--Agler class},
introduced in \cite{A}, will be denoted by
$\mathcal{SA}_d(\Hspace{U},\Hspace{Y})$. It consists of analytic
$\mathcal{L}(\Hspace{U},\Hspace{Y})$-valued functions on ${\mathbb
D}^d$ such that
\begin{equation}\label{eq:Agler-norm}
\|f\|_{\mathcal{A}}:=\sup_T\| f(T) \| \le 1,
 \end{equation}
where the supremum is taken over all $d$-tuples
$T=(T_1,\ldots,T_d)$ of commuting strict contractions on a common
Hilbert space. In the scalar case
$\Hspace{U}=\Hspace{Y}=\mathbb{C}$ and in the case $\Hspace
U=\Hspace Y$, we will use respective shortcuts $\mathcal{S}_d$,
$\mathcal{SA}_d$, and $\mathcal{S}_d(\Hspace U)$,
$\mathcal{SA}_d(\Hspace U)$.

For a bounded  analytic function $f\colon {\mathbb D}^d\to \mathcal{L}(\Hspace{U},\Hspace{Y})$,
the \emph{von Neumann inequality} is the inequality between its Schur and Schur--Agler norms:
\begin{equation}\label{eq:vN} \| f \|_\mathcal{A} \le \| f \|_\infty. \end{equation}
It is valid when $d=1$ \cite{vN} and $d=2$ \cite{Ando-pair}, and
not always valid when $d\ge3$ \cite{Varo,CD,Hol01}. Thus a Schur
function $f$ is Schur--Agler if and only if \eqref{eq:vN} holds.
One has the inclusion
$\mathcal{SA}_d(\Hspace{U},\Hspace{Y})\subseteq\mathcal{S}_d(\Hspace{U},\Hspace{Y})$.
The two classes coincide when $d=1$ and $d=2$, and the inclusion
is proper when $d\ge3$. See, e.g., \cite{GKVVW} for details.

A $d$-variable polynomial is said to be \emph{stable} if it has no
zeros in $\overline{\mathbb{D}}^d$, and \emph{semi-stable} if it
has no zeros in ${\mathbb{D}}^d$. A rational function in $\mc
S_d(\mathbb{C}^N)$ is said to be \emph{inner} if its radial limits
are unitary (unimodular, in the scalar case) almost everywhere on
the $d$-torus. Every scalar-valued rational inner function is
necessarily of the form $f(z)=z^n\bar p(1/z)/p(z)$ for some
$n\in\mathbb N_0^d$ and a semi-stable polynomial $p$ \cite[Theorem
5.5.1]{Rudin}, where $\bar{p}(z):=\overline{p(\bar{z})}$. A
rational inner function $f\in\mc S_d$ is said to have a {\it
transfer-function realization} (of order $m\in\mathbb N_0^d$) if
there exists a unitary matrix
$$U=\left[
\begin{matrix} A& B \cr C& D \end{matrix} \right] \in
{\mathbb C}^{(1+|m|)\times (1+|m|)}$$ so that
\begin{equation}\label{eq:tf}
f(z) = A + BZ_m (I-DZ_m)^{-1}C.
\end{equation}
Such a realization for a scalar-valued rational inner $f$ exists if and only if
$f\in\mc{SA}_d$ \cite{A},\cite[Theorem 2.9]{Knese2011}.

In Section \ref{wedge}, we explore the Schur--Agler class in the
context of exterior products, proving, in particular, that if $S$
is a matrix-valued Schur--Agler function, then so are its
determinant $\det S$ and permanent $\per S$.

Following \cite{KneseSymm}, we will say that a semi-stable
polynomial $p$ is {\it an Agler denominator} if the rational inner
function $z^{\deg p}\bar{p}(1/z)/p(z)$ is Schur--Agler. Extending
this notion, we will call a semi-stable polynomial $p$ an {\it
eventual Agler denominator} of order $n\in {\mathbb N}_0^d$ if
$z^n \bar{p}(1/z)/p(z)$ is Schur--Agler.

Representations \eqref{eq:repr} may allow for a fresh approach to
the study of the multivariable von Neumann inequality
\eqref{eq:vN}. In Section \ref{Ad}, we examine the discrepancy
between the Schur and Schur--Agler classes via (eventual) Agler
denominators. It is shown that (i) not every (semi-)stable
polynomial is an Agler denominator, (ii) if $q$ is a polynomial in
$\mc{SA}_d$  with $q(0)=0$, then $1-q$ admits a representation
\eqref{eq:repr} for some $n \in {\mathbb N}_0^d$ and $K$ a
contraction, and (iii) (building on results in Section
\ref{wedge}) if $p$ is representable in the form \eqref{eq:repr}
for some $n \in {\mathbb N}_0^d$ and contractive $K$, then $p$ is
an eventual Agler denominator of order $n$. As a corollary, we
deduce that every semi-stable linear polynomial is an Agler
denominator, thus solving a problem suggested in \cite{KneseSymm}.
To illustrate a possible advantage of our approach,
 we compare
 a minimal determinantal representation \eqref{eq:repr} to
a minimal transfer-function realization \eqref{eq:tf}
 in Remark \ref{rem:linear}.

In Section \ref{sec6}, we revisit the Kaijser--Varopoulous--Holbrook example to build a family of polynomials
in $\mc S_d\setminus\mc{SA}_d$, for every odd $d\ge3$.
If $d=3$, this leads to a slightly improved bound for the von Neumann constant.

%It is clear that a polynomial $p$ is semi-stable (stable) as long
%as it admits \eqref{eq:repr} with $K$ a contraction (strict
%contraction). The question of whether the converse is true is open
%(see Remark \ref{rem:question}).

\section{Unconstrained determinantal representations}\label{sec:affine-det-repr}

\begin{thm}\label{thm:k-repr}
Every $p\in\mathbb{C}[z_1,\ldots,z_d]$, with $p(0)=1$,
 admits a representation \eqref{eq:repr} for some $n\in\mathbb{N}_0^d$ and
some $K\in\mathbb{C}^{|n|\times |n|}$.
\end{thm}
Note that the $d$-tuple $n$ is not prescribed in the statement,
although a bound on $n$ will be deduced in the proof. If $n$ is
specified, as in the Principal Minor Assignment Problem, Theorem
\ref{thm:k-repr} ensures the solvability of the following problem
for sone $n$ (see also Remark \ref{rem:PMRP}).
\medskip

\noindent \textbf{The Principal Minor Relation Problem.} Let $m\in\mathbb{N}_0^d$, and let $p_k$, $0\le k\le m$, be a collection of complex numbers. Given $n=(n_1,\ldots,n_d)\in\mathbb{N}_0^d$, $n\ge m$,
find a matrix $K\in\mathbb{C}^{|n|\times|n|}$ whose principal minors $K[\alpha_1\cup\cdots\cup\alpha_d]$, indexed
by $\alpha_1\subseteq\{1,\ldots,n_1\}$,
$\alpha_2\subseteq\{n_1+1,\ldots,n_1+n_2\}$, \ldots,
$\alpha_d\subseteq\{n_1+\cdots+n_{d-1}+1,\ldots,|n|\}$,
satisfy the relations
\begin{equation}\label{eq:PMRP}
(-1)^{|k|}\sum_{|\alpha_i|=k_i,\ i=1,\ldots,d}\det
K[\alpha_1\cup\cdots\cup\alpha_d]=p_k,\quad  0\le k\le m.
\end{equation}
When $m=n=(1,\ldots,1)$, this is the classical Principal Minor Assignment Problem mentioned in Section \ref{sec:Intro}.

The following standard result (see, e.g., \cite[Theorem
3.1.1]{Prasolov}) will occasionally be used.
\begin{lem}\label{lem:Schur-compl}
Let $P=\begin{bmatrix} A & B\\
C & D
\end{bmatrix}$ be a block matrix with square matrices $A$ and $D$.
If $\det A\neq 0$, then $\det P=\det A\det(D-CA^{-1}B)$.
Similarly, if $\det D\neq 0$, then $\det P=\det D\det
(A-BD^{-1}C)$.
\end{lem}

The proof of Theorem \ref{thm:k-repr} will be based on the next
two lemmas.
\begin{lem}\label{lem:prod}
 For every $q\in\mathbb{C}^{a\times b}[z_1,\ldots,z_d]$, there exist natural numbers $s_0=a$, $s_1$,
\ldots, $s_{t-1}$, $s_t=b$, matrices $C_i\in\mathbb{C}^{s_i\times
s_{i+1}}$, and diagonal $s_i\times s_i$ matrix functions $L_i$
with the diagonal entries in $\{1,z_1,\ldots,z_d\}$, such that
\begin{equation}\label{eq:prod}
q(z)=C_0L_1(z)\cdots C_{t-1}L_t(z)C_t.
\end{equation}
The factorization can be chosen so that $t=\tdeg q$.
\end{lem}
\begin{proof}
We apply induction on $t$. If $t=0$, then \eqref{eq:prod} holds
trivially with $C_0=q(z)$. Suppose a representation
\eqref{eq:prod} exists for every matrix polynomial in $z_1$,
\ldots, $z_d$ of total degree $t-1$. Then a polynomial
$q\in\mathbb{C}^{a\times b}[z_1,\ldots,z_d]$ of total degree $t$
can be represented in the form
\begin{multline*}
q(z)=q_0+z_1q_1(z)+\cdots +z_dq_d(z)\\
=\begin{bmatrix} q_0 & q_1(z) & \ldots & q_d(z)
\end{bmatrix}
\begin{bmatrix}
I_b & 0        & \ldots & 0\\
0   & z_1I_b   & \ddots & \vdots\\
\vdots & \ddots & \ddots & 0\\
0 & \ldots & 0 & z_dI_b
\end{bmatrix}
\begin{bmatrix}
I_b\\
I_b\\
\vdots\\
I_b
\end{bmatrix},
\end{multline*}
where $\begin{bmatrix} q_0 & q_1(z) & \ldots & q_d(z)
\end{bmatrix}\in\mathbb{C}^{a\times (d+1)b}[z_1,\ldots,z_d]$ is a polynomial of total degree $t-1$. By assumption,
we have
$$\begin{bmatrix} q_0 & q_1(z) & \ldots & q_d(z)
\end{bmatrix}=C_0L_1(z)\cdots C_{t-2}L_{t-1}(z)C_{t-1},$$
which gives $q(z)=C_0L_1(z)\cdots C_{t-1}L_t(z)C_t$, with
$$L_t=\begin{bmatrix}
I_b & 0        & \ldots & 0\\
0   & z_1I_b   & \ddots & \vdots\\
\vdots & \ddots & \ddots & 0\\
0 & \ldots & 0 & z_dI_b
\end{bmatrix},\quad C_t=\begin{bmatrix}
I_b\\
I_b\\
\vdots\\
I_b
\end{bmatrix}.$$
\end{proof}

\begin{lem}\label{lem:det}
Let $A_i\in\mathbb{C}^{s_i\times s_{i+1}}$, $i=0,\ldots,t-1$, and
$A_t\in\mathbb{C}^{s_t\times s_0}$, where $s_0=a$. Then
\begin{equation}\label{eq:det}
\det\begin{bmatrix}
I_a & -A_0 & 0        & \ldots & 0\\
0   & I_{s_1} & \ddots  & \ddots & \vdots\\
\vdots & \ddots & \ddots & \ddots & 0\\
0 &  & \ddots & \ddots & -A_{t-1}\\
-A_t & 0 &\ldots & 0 & I_{s_t}
\end{bmatrix}=\det(I_a-A_0\cdots A_t).
\end{equation}
\end{lem}
\begin{proof}
We apply induction on $t\ge 1$. For $t=1$,  Lemma
\ref{lem:Schur-compl} gives
$$\det\begin{bmatrix}
I_a & -A_0\\
-A_1 & I_{s_1}
\end{bmatrix}=\det(I_a-A_0A_1).$$
 Suppose \eqref{eq:det} holds for $t-1$ in the place of $t$. Then,
 again by Lemma
\ref{lem:Schur-compl},
\begin{multline*}
\det\begin{bmatrix}
I_a & -A_0 & 0        & \ldots & 0\\
0   & I_{s_1} & \ddots  & \ddots & \vdots\\
\vdots & \ddots & \ddots & -A_{t-2} & 0\\
0 & \ldots & 0 & I_{s_{t-1}} & -A_{t-1}\\
-A_t & 0 &\ldots & 0 & I_{s_t}
\end{bmatrix}\\
=\det\left(\begin{bmatrix}
I_a & -A_0 & 0        & \ldots & 0\\
0   & I_{s_1} & \ddots  & \ddots & \vdots\\
\vdots & \ddots & \ddots & \ddots & 0\\
\vdots & & \ddots & \ddots & -A_{t-2}\\
0 & \ldots &\ldots & 0 & I_{s_{t-1}}
\end{bmatrix}-\begin{bmatrix}
0  \\
\vdots \\
0\\
-A_{t-1}
\end{bmatrix}\begin{bmatrix}
-A_t & 0 &\ldots & 0
\end{bmatrix}\right)\\
=\det\begin{bmatrix}
I_a & -A_0 & 0        & \ldots & 0\\
0   & I_{s_1} & \ddots  & \ddots & \vdots\\
\vdots & \ddots & \ddots & \ddots & 0\\
0 &  & \ddots & \ddots & -A_{t-2}\\
-A_{t-1}A_t & 0 &\ldots & 0 & I_{s_{t-1}}
\end{bmatrix}\\
=\det(I_a-A_0\cdots A_{t-1}A_t).
\end{multline*}
\end{proof}

\begin{proof}[Proof of Theorem \ref{thm:k-repr}]
Applying Lemma \ref{lem:prod} to $q=1-p$, we obtain
$$p(z)=1-C_0L_1(z)\cdots C_{t-1}L_t(z)C_t$$
(here $a=b=1$). So, by Lemma \ref{lem:det},
$$p(z)=\det(I_N-Q(z)),$$
where $N=1+s_1+\cdots+s_t$ and
$$Q(z)=\begin{bmatrix}
0 & C_0L_1(z) & 0 & \ldots  & 0\\
\vdots & \ddots & \ddots & \ddots & \vdots\\
\vdots &   & \ddots     & \ddots & 0 \\
0 & \ldots & \ldots & 0 & C_{t-1}L_t(z)\\
C_t & 0 & \ldots & \ldots & 0
\end{bmatrix}.$$
Since $Q(z)$ factors as $C\cdot L(z)$, where
\begin{equation}\label{eq:CL}
C=\begin{bmatrix}
0 & C_0 & 0 & \ldots  & 0\\
\vdots & \ddots & \ddots & \ddots &\vdots\\
\vdots &   &   \ddots   & \ddots & 0 \\
0 & \ldots & \ldots & 0 & C_{t-1}\\
C_t & 0  & \ldots & \ldots & 0
\end{bmatrix},\ L(z)=\begin{bmatrix}
1 & 0 &  \ldots  & 0\\
0 & L_1(z) & \ddots & \vdots\\
\vdots &  \ddots & \ddots & 0\\
0 & \ldots & 0 & L_t(z)
\end{bmatrix},\end{equation}
it may be written in the form
$$Q(z)=TGT^{-1}\cdot T\begin{bmatrix}
I_{N-|n|} & 0\\
0 & Z_n
\end{bmatrix}T^{-1}$$ where $G=T^{-1}CT\in\mathbb{C}^{N\times N}$ and $T$ is a permutation matrix. Representing $G$
as a $2\times 2$ block matrix, we obtain that
$$p(z)=\det\begin{bmatrix}
I_{N-|n|}-G_{11} & -G_{12}Z_n\\
-G_{21} & I_{|n|}-G_{22}Z_n
\end{bmatrix}.$$
Hence $\det(I_{N-|n|}-G_{11})=p(0)=1$ and, in particular,  the
matrix $I_{N-|n|}-G_{11}$ is invertible. Therefore, by Lemma
\ref{lem:Schur-compl},
$$p(z)=\det(I_{|n|}-G_{22}Z_n-G_{21}(I_{N-|n|}-G_{11})^{-1}G_{12}Z_n)=\det(I_{|n|}-KZ_n)$$
 with $K=G_{22}+G_{21}(I_{N-|n|}-G_{11})^{-1}G_{12}$.
\end{proof}

\section{Constrained determinantal representations}\label{sec:constr-det-repr}

We will now look into the existence of a determinantal representation \eqref{eq:repr} with a norm constraint on the matrix $K$.
First, we give norm-constrained versions of Lemma \ref{lem:prod} and Theorem \ref{thm:k-repr}.
\begin{lem}\label{lem:prod-constr}
For every polynomial $q\in
\mathcal{SA}_d(\mathbb{C}^b,\mathbb{C}^a)$, a factorization
\eqref{eq:prod} exists with constant contractive matrices $C_i$,
$i=0,\ldots,t$, where $t\ge \tdeg q$.
\end{lem}
\begin{proof}
Let $q\in \mathcal{SA}_d(\mathbb{C}^b,\mathbb{C}^a)$ be a
polynomial. By \cite[Corollary 18.2]{Pau}, $q$ can be written as a
product of constant contractive matrices and diagonal matrices
with monomials on the diagonal. Every such diagonal matrix is, in
turn, a product of matrices $L_i$ as in \eqref{eq:prod}
(interlacing with $C_i=I$).
\end{proof}
\begin{thm}\label{thm:q-matrix}
Let $p$ be a polynomial of the form $p(z)=\det(I_N-q(z))$, where
$q$ is a  Schur--Agler polynomial with matrix coefficients,
i.e., $q\in\mathbb{C}^{N\times N}[z_1,\ldots,z_d]$ and
$q\in\mathcal{SA}_d(\mathbb{C}^N)$. If $p(0)=1$, then \eqref{eq:repr} holds
with $K$ a contraction.
\end{thm}
\begin{proof}
By Lemma \ref{lem:prod-constr}, the matrices $C_i$ in the
factorization \eqref{eq:prod} can be chosen contractive. Then the
 matrix $G$ as in the proof of Theorem \ref{thm:k-repr} is also contractive, and by the standard closed-loop
 mapping argument, $K$ is contractive as well.
 For reader's convenience, we include this argument.

 Given $u\in\mathbb{C}^{|n|}$, the vector equation
 $$\begin{bmatrix}
 G_{11} & G_{12}\\
 G_{21} & G_{22}
 \end{bmatrix}\begin{bmatrix}
 x\\
 u
 \end{bmatrix}=\begin{bmatrix}
 x\\
 y
 \end{bmatrix}$$
in $x\in\mathbb{C}^{N-|n|}$ and $y\in\mathbb{C}^{|n|}$ has a
unique solution
 $x=(I_{N-|n|}-G_{11})^{-1}G_{12}u$, $y=(G_{22}+G_{21}(I_{N-|n|}-G_{11})^{-1}G_{12})u=Ku$. Since $G$ is a
 contraction, we have $\|x\|^2+\|y\|^2\le\|x\|^2+\|u\|^2$, i.e., $\|y\|\le\|u\|$. Since $u\in\mathbb{C}^{|n|}$ is
 arbitrary, $K$ is a contraction as claimed.
\end{proof}
\begin{cor}\label{cor:q-scalar}
Let $p\in\mathbb{C}[z_1,\ldots,z_d]$, with $p(0)=1$, be such that
$q=1-p\in\mathcal{SA}_d$, then \eqref{eq:repr} holds with $K$ a
contraction.
\end{cor}
\begin{rem}\label{rem:converse-false}
The converse to Corollary \ref{cor:q-scalar} is false. Indeed, if $d=1$ and $\sqrt{2}-1<|a|\le 1$, then $p(z)=(1-az)^2$
satisfies \eqref{eq:repr} with $K=aI_2$, obviously a contraction. However,
$\|1-p\|_\mathcal{A}=\|1-p\|_\infty=2|a|+|a|^2>1$.
\end{rem}

Define the \emph{stability radius} $s(p)$ of a $d$-variable polynomial $p$ to be
$$s(p)= \sup\bigg\{ r>0\colon p(z)\ne0,\ z\in r{\mathbb D}^d\bigg\}.$$
Clearly, $p$ is semi-stable if $s(p)\ge 1$, and stable if
$s(p)>1$. It is easy to see that $\|K\|\ge1/s(p)$ whenever $p$
admits \eqref{eq:repr}.
\begin{rem} With respect to a given subalgebra $\Delta\subseteq {\mathbb C}^{d\times d}$, the structured singular value $\mu_{\Delta}(K)$
of a matrix $K\in\mathbb C^{d\times d}$ is defined to be
$$ \mu_{\Delta}(K) := \left(\inf \big\{ \| Z \|\colon Z \in \Delta \ {\rm and } \ \det (I-KZ) = 0 \big\} \right)^{-1}.$$
The theory of structured singular values was introduced in \cite{Doyle} to analyze linear systems with structured uncertainties; for an
overview, see for instance \cite[Chapter 10]{Zhou}. If $p$ satisfies \eqref{eq:repr} and
$\Delta=\{ Z_n=\bigoplus_{i=1}^dz_iI_{n_i}\colon \ z \in {\mathbb C}^d \}$, we recognize that $\mu_{\Delta}(K)=1/s(p)$.
\end{rem}
The next theorem gives a way of constructing a representation \eqref{eq:repr} with a
certain upper bound on the norm of $K$.
\begin{thm}\label{thm:norm-bds}
Given a polynomial $p(z)=1+\sum_{k\in S}p_kz^k$, where
$S\subseteq\mathbb{N}_0^d\setminus\{0\}$ and the coefficients
$p_k, k\in S,$ are nonzero, let $t=\tdeg p$, $n=\sum_{k\in S}k$,
and $\beta=\left(\sum_{k\,\in\,S}|p_k|\right)^{\frac{1}{t+1}}$.
Then $p$ admits a representation \eqref{eq:repr} with
$K\in\mathbb{C}^{|n|\times|n|}$, and
\begin{equation}\label{eq:k-bound}
\|K\|\le\beta\max\left\{\sqrt{(\beta^2-1)(1+\beta+\cdots+\beta^{\kappa-1})^2+1},\ 1\right\}
\end{equation}
for some integer $\kappa$, $1\le\kappa\le t$.
\end{thm}
\begin{rem}\label{rem:beta} If $\beta\le 1$ and $s(p)=1/\beta$, which is the case for semi-stable linear polynomials,
the norm bound asserted in the theorem is sharp. In general, it is
not sharp, even in the univariate case.
\end{rem}
\begin{rem}\label{rem:PMRP}
Theorem \ref{thm:norm-bds} implies that the Principal Minor
Relation Problem \eqref{eq:PMRP} with data $\{p_k\neq 0\colon k\in
S\}$ is solvable for $n=\sum_{k\in S}k$.
\end{rem}
\begin{proof}[Proof of Theorem \ref{thm:norm-bds}]
Form the matrices
$$C_0=-\beta^{\frac{1-t}{2}}\ \underset{k\in
S}{\row}\bigg[|p_k|^{\frac12}\bigg],\ C_1=\ldots=C_{t-1}=\beta
I_{|S|},\ C_t=\beta^{\frac{1-t}{2}}\ \underset{k\in
S}\col\bigg[\frac{p_k}{|p_k|^{\frac12}}\bigg],$$ all of equal norm
$\beta$. Relative to the standard ordering of factors,
$$z^k=\underset{k_1\ {\rm times}}{\underbrace{z_1\cdots z_1}}\cdot \underset{k_2\ {\rm times}}
{\underbrace{z_2\cdots z_2}}\cdot\ldots\cdot \underset{k_d\ {\rm
times}}{\underbrace{z_d\cdots z_d}} ,$$ write each monomial $z^k$, $k\in S$, as an expanded product
$z^k=z_{i_1(k)}\cdots z_{i_t(k)}$, where $z_{i_j(k)}\ne1$, for $1\le j\le|k|$, and $z_{i_j(k)}=1$,
for $|k|+1\le j\le t$. Let
$$L_j(z)=\underset{k\in S}{\diag}[z_{i_j(k)}],\quad j=1,\ldots,t,$$
and observe that $L_1(z)$ contains no unit entries by construction.
It is then easy to check that \eqref{eq:prod} holds for $q=1-p$.
Thus, by Lemma \ref{lem:det}, we obtain
\begin{align*}
p(z)&=\det\left(I_{1+|S|t}-
\begin{bmatrix}
0 & C_0L_1(z) & 0 & \ldots  & 0\\
\vdots & \ddots & \ddots & \ddots & \vdots\\
\vdots &   & \ddots     & \ddots & 0 \\
0 & \ldots & \ldots & 0 & C_{t-1}L_t(z)\\
C_t & 0 & \ldots & \ldots & 0
\end{bmatrix}\right)\\
&=\det(I_{1+|S|t}-C\cdot L(z)),
\end{align*} where $C$ and $L(z)$ are as in \eqref{eq:CL}.

Using an appropriate permutation $T$, we can bubble-sort $L(z)$ so
that all diagonal ones are stacked in the left upper corner block:
$$TL(z)T^{-1}=\begin{bmatrix}
I_\ell & 0\\
0 & Z_{n}
\end{bmatrix},
$$ where $\ell=1+|S|t-|n|$.
Then, partitioned accordingly,
$$TCT^{-1}=\begin{bmatrix}
G_{11} & G_{12}\\
G_{21} & G_{22}
\end{bmatrix}=:G$$
has the following structure:
$$G=\begin{bmatrix}
\begin{matrix}
0 & 0 & 0 & \ldots  &  0\\
\vdots & \ddots &  * & \ddots &   \vdots\\
\vdots &   &   \ddots   & \ddots &  0 \\
0 & \ldots & \ldots & 0 &   *\\
* & 0  & \ldots & \ldots  & 0
\end{matrix} & \ & \begin{matrix}
* & 0 & 0 & \ldots  &  0\\
0 & 0 &  * & \ddots &   \vdots\\
\vdots & \ddots  &   \ddots   & \ddots &  0 \\
\vdots &  & \ddots & \ddots &   *\\
0 & \ldots  & \ldots & 0  & 0
\end{matrix}\\
& & \\
\begin{matrix}
0 & * & 0 & \ldots  &  0\\
\vdots & \ddots &  \ddots & \ddots &   \vdots\\
\vdots &   &   \ddots   & \ddots &  0 \\
0 & \ldots & \ldots & 0 &   *\\
* & 0  & \ldots & \ldots  & 0
\end{matrix} & \ & \begin{matrix}
0 & * & 0 & \ldots  &  0\\
\vdots & \ddots &  \ddots & \ddots &   \vdots\\
\vdots &   &   \ddots   & \ddots &  0 \\
\vdots &  &  & \ddots &   *\\
0 & \ldots  & \ldots & \ldots  & 0
\end{matrix}
\end{bmatrix}.$$
We observe that $G_{11}$ is nilpotent of index $\kappa$, where
$\kappa$ is the number of nonzero blocks $L_{i}(0)$ (counting
$L_0\equiv 1$). Since, $L_1(0)=0$, we necessarily have
$1\le\kappa\le t$. Thus, we obtain that
$$p(z)=\det(I_{1+|S|t}-C\cdot L(z))=\det\Big(I_{1+|S|t}-G\cdot \begin{bmatrix}
I_\ell & 0\\
0 & Z_{n} \end{bmatrix}\Big)=\det(I_{|n|}-KZ_n),$$ where
$K=G_{22}+G_{21}(I-G_{11})^{-1}G_{12}$ by the same argument as in
the proof of Theorem \ref{thm:q-matrix}.

The norm bound on $K$ is obtained as follows. For a fixed $u\in\mathbb{C}^{|n|}$, the solution to the vector equation
 $$\begin{bmatrix}
 G_{11} & G_{12}\\
 G_{21} & G_{22}
 \end{bmatrix}\begin{bmatrix}
 x\\
 u
 \end{bmatrix}=\begin{bmatrix}
 x\\
 y
 \end{bmatrix}$$
in $x\in\mathbb C^{\ell}$ and $y\in\mathbb C^{|n|}$ is given by
$$x=(I_{\ell}-G_{11})^{-1}G_{12}u=(I_\ell+G_{11}+\cdots+G_{11}^{\kappa-1})G_{12}u,\quad
y=Ku.$$ Observing that $\|G\|=\beta$, we have
$$\|x\|^2+\|y\|^2\le\beta^2(\|x\|^2+\|u\|^2).$$
If $\beta\le1$, then $\|y\|\le\beta\|u\|$ and thus $\|K\|\le\beta$. If $\beta>1$, then the estimate
$$\|x\|\le(1+\beta+\cdots+\beta^{\kappa-1})\beta\|u\|$$
implies that
\begin{multline*}
\|y\|^2\le(\beta^2-1)\|x\|^2+\beta^2\|u\|^2
\le\beta^2\left((\beta^2-1)(1+\beta+\cdots+\beta^{\kappa-1})^2+1\right)\|u\|^2,
\end{multline*}
which yields \eqref{eq:k-bound}.
\end{proof}
\begin{rem}\label{rem:question}
Given a $d$-variable polynomial $p$, $p(0)=1$, and a $d$-tuple
$n\ge\deg p$ such that \eqref{eq:repr} holds, one may consider the
set $\mc K_n(p)$ of $|n|\times|n|$ matrices $K$ such that
$\det(I_{|n|}-KZ_n)=p(z)$. It is then of interest to determine the
constant
  $$\alpha(p):=\inf_{n}\min_{K\in\mathcal K_n(p)}\|K\|.$$
  In particular, it is unclear whether $\alpha(p)<1$
  ($\alpha(p)\le 1$) for $p$ stable (semi-stable).
\end{rem}

\section{The Schur--Agler class and wedge powers}\label{wedge}

We will now examine the Schur--Agler norm of tensor and exterior
products of operator-valued functions. The results are preceded by
some definitions. For a background on tensor and exterior algebras
see, e.g., \cite{Bhatia, Flanders, Greub}.

Let $\mc V^{\otimes k}$ be the $k$-fold tensor power of a vector
space $\mc V$. The $k$-th antisymmetric tensor power $\mc
V^{\wedge k}$ of $\mc V$ may be viewed as a subspace of $\mc
V^{\otimes k}$, generated by elementary antisymmetric tensors
$$v_1\wedge\ldots\wedge v_k=\sum_{\s}({\rm sign}\
\s)v_{\s(1)}\otimes\ldots\otimes v_{\s(k)},$$ where the summation
is taken over all permutations $\s$ of $1, 2, \ldots, k$.

Given a linear map $A:\mc U\to\mc V$ of vector spaces, the linear
operator $A^{\wedge k}:\ \mc U^{\wedge k}\to\mc V^{\wedge k}$,
determined by the equalities
$$A^{\wedge k}(u_1\wedge\ldots\wedge u_k)=Au_1\wedge\ldots \wedge Au_k,$$
is the compression $\pi_{\mc V^{\wedge k}} A^{\otimes k}\big|_{\mc
U^{\wedge k}}$ of the tensor power $A^{\otimes k}\colon \mathcal
U^{\otimes k}\to\mathcal V^{\otimes k}$. Here $\pi_M$ denotes the
orthogonal projection onto a subspace $M$.

If $e_1, \ldots, e_n$ form a basis for $\mc V$, then
$e_{i_1}\wedge\ldots\wedge e_{i_k},\ 1\le i_1<\ldots<i_k\le n,$
form a basis for $\mc V^{\wedge k}$ of cardinality $\binom nk$.
Relative to a choice of bases for  $\mc U$ and $\mc V$, the matrix
entry for  $A^{\wedge k}$ in row-column position $((i_1, \ldots,
i_k), (j_1, \ldots, j_k))$ is the minor of the matrix of $A$ built
from rows $i_1, \ldots, i_k$ and columns $j_1, \ldots, j_k$.

If $\mathcal U$ and $\mathcal V$ are normed vector spaces and if
$S(z)=\sum_{r\in\mathbb N_0^d} S_r z^r$ is a power series with
coefficients in $\mathcal{L(U,V)}$, then, for any tuple $T=(T_1,
\ldots, T_d)$ of commuting operators on some normed vector space
$\mathcal H$, we may consider the operator
$$S(T)=\sum_{r\in\mathbb N_0^d}S_r\otimes T^r$$ acting from $\mc U\otimes\mc H$ to
$\mc V\otimes\mc H$, provided the series converges. More
generally, starting with power series $S_j(z)=\sum_{r\in\mathbb
N_0^d} (S_j)_r z^r$, $j=1,\ldots,k,$ with coefficients in
$\mathcal{L(U,V)}$,
 and $T=(T_1, \ldots, T_d)$ as above, form the
operator
$$\left(S_1\otimes\cdots\otimes S_k\right)(T)=\sum_{r_1,\ldots, r_k\in\mathbb
N^d_0}\left(\bigotimes_{j=1}^k(S_j)_{r_j}\right)\otimes
T^{r_1+\ldots+r_k},$$ and its compression
$$\left(S_1\wedge\cdots\wedge S_k\right)(T)=(\pi_{\mc V^{\wedge k}}\otimes I_{\mc
H})(S_1\otimes\cdots\otimes S_k)(T)\big|_{\mc U^{\wedge
k}\otimes\mc H}.$$ The objective of this section is to establish
the following theorem.
\begin{thm}
Let $\mc U$ and $\mc V$ be Hilbert spaces and let $S_1, \ldots,
S_k$ belong to $\mc S\mc A_d(\mc U, \mc V)$. Then
$S_1\otimes\cdots\otimes S_k$ belongs to $\mc S\mc A_d(\mc
U^{\otimes k}, \mc V^{\otimes k})$ and $S_1\wedge\cdots\wedge S_k$
belongs to $\mc S\mc A_d(\mc U^{\wedge k}, \mc V^{\wedge k})$.
\end{thm}
\begin{proof}
Let $T=(T_1, \ldots, T_d)$ be a tuple of commuting strict
contractions on some Hilbert space $\mc H$. Then the mapping
\begin{align*}
\left(\bigotimes_{j=1}^kS_j\right)(T)&=\sum_{r_1,\ldots,r_k\in\mathbb
N^d_0}\left(\bigotimes_{j=1}^k(S_j)_{r_j}\right)\otimes T^{r_1+\ldots+r_k}\\
&=\prod_{j=1}^k\sum_{r_j\in\mathbb N_0^d}I_{\mc
V}\otimes\cdots\otimes I_{\mc
V}\otimes(S_j)_{r_j}\otimes I_{\mc V}\otimes\cdots\otimes I_{\mc V}\otimes T^{r_j}\\
&=\prod_{j=1}^k\bigg(I_{\mc V}\otimes\cdots\otimes I_{\mc
V}\otimes S_j\otimes I_{\mc V}\otimes\cdots\otimes I_{\mc
V}\bigg)(T)
\end{align*}
is contractive as a product of contractive factors. Hence
$\bigotimes_{j=1}^kS_j$ belongs to $\mc S\mc A_d(\mc U^{\otimes
k}, \mc V^{\otimes k})$. Consequently,
$$\|S_1\wedge\ldots\wedge S_k\|_{\mc A}\le\|S_1\otimes\ldots\otimes S_k\|_{\mc
A}\le1,$$
which gives the second assertion.
\end{proof}
\begin{cor}\label{detcor}
Let $S$ be a $n\times n$ matrix-valued Schur--Agler function,
i.e., $S\in\mathcal{SA}_d(\mathbb{C}^n)$. Then, for every $k=1,
\ldots, n$, the $k$-th compound matrix-valued function of $S$ is
also Schur--Agler. In particular, $\det S(z)$ is a Schur--Agler
function.
\end{cor}
\begin{proof}
The matrix of $S^{\wedge k}$ is the $k$-th compound matrix of $S$.
The case $k=n$ corresponds to $\det S(z)$.
\end{proof}

Similarly, in the setting of $k$-th symmetric tensor powers, one
may consider the operators $(S_1\vee\cdots\vee S_k)(T)$. The proof
of the following theorem is omitted as it parallels the preceding
development.
\begin{thm}
Let $\mc U$ and $\mc V$ be Hilbert spaces and let $S_1, \ldots,
S_k$ belong to $\mc S\mc A_d(\mc U, \mc V)$. Then
$S_1\vee\cdots\vee S_k$ belongs to $\mc S\mc A_d(\mc U^{\vee k},
\mc V^{\vee k})$.
\end{thm}
\begin{cor}
Let $S$ be a $n\times n$ matrix-valued Schur--Agler function,
i.e., $S\in\mathcal{SA}_d(\mathbb{C}^n)$. Then, for every $k=1,
\ldots, n$, the $k$-th permanental compound matrix-valued function
of $S$ is also Schur--Agler. In particular, the permanent of a
Schur--Agler function is also Schur--Agler.
\end{cor}
We note that a permanental analog of \eqref{eq:repr} features in
\cite{Branden}.

\section{Agler denominators}\label{Ad}

We are in a position to discuss (eventual) Agler denominators and stability in relation to \eqref{eq:repr}.
It will first be shown that  there exist stable polynomials in three or more variables that are not Agler denominators.

\begin{ex} \rm
Let $p(z)$ be a $d$-variable polynomial, with $\|p\|_\infty=1$ and
multi-degree $m$,
violating the von Neumann inequality \eqref{eq:vN}.
Let there exist a tuple $T=(T_1, \ldots, T_d)$ of commuting
contractions such that $T^k=T_1^{k_1}T_2^{k_2}\cdots T_d^{k_d}=0$,
for some $k\in\mathbb N_0^d$, and $\|p(T)\|>1$. Examples of such a
scenario can be found in  \cite{Varo,Hol01,CD}; see also Section
\ref{sec6}.

For $0<r<1$, the polynomial $q(z)=1+rz^{k+m}\bar p(1/z)$ is stable
and so the rational function
$$f(z)=\frac{z^{k+m}+rp(z)}{1+rz^{k+m}\bar p(1/z)}$$ is inner. However,
since $f(T)=rp(T)$, $f$ does not belong to $\mc{SA}_d$ whenever
$r>1/\|p(T)\|$.
In particular, if the multi-degree of $z^m\bar{p}(1/z)$ is also $m$, then
$f(z)=z^{k+m}\bar q(1/z)/q(z)$, so that $q$ is not an Agler denominator.

% Let $p(z)$ be a $d$-variable polynomial, with $\|p\|_\infty=1$,
%violating the von Neumann inequality \eqref{eq:vN} and such that
%$p$ and $z^{\deg p}\bar{p}(1/z)$ have the same multi-degree $m$.
%Let there exist a tuple $T=(T_1, \ldots, T_d)$ of commuting
%contractions such that $T^k=T_1^{k_1}T_2^{k_2}\cdots T_d^{k_d}=0$,
%for some $k\in\mathbb N_0^d$, and $\|p(T)\|>1$. Examples of such a
%scenario can be found in  \cite{Varo,Hol01,CD}; see also Section
%\ref{sec6}.
%
%For $0<r<1$, the polynomial $q(z)=1+rz^{k+m}\bar p(1/z)$ is stable
%and so the rational function
%$$f(z)=z^m\ \frac{\bar q(1/z)}{q(z)}=\frac{z^{k+m}+rp(z)}{1+rz^{k+m}\bar p(1/z)}$$ is inner. However,
%since $f(T)=rp(T)$, $f$ does not belong to $\mc{SA}_d$ whenever $r>1/\|p(T)\|$.

To give a concrete example, we specialize to the Kaijser--Varopoulos--Holbrook setting. The polynomial
$$p(z_1, z_2, z_3)=\frac15\bigg(z_1^2+z_2^2+z_3^2-2z_1z_2-2z_2z_3-2z_3z_1\bigg)$$
satisfies $\|p\|_\infty=1$, and there exist commuting contractions $T_1, T_2, T_3$ such that
$\|p(T_1, T_2, T_3)\|=6/5$ and $T_1T_2T_3=0$. The corresponding rational inner function
$$f(z_1, z_2, z_3)=\frac{z_1^3z_2^3z_3^3+\frac
r5(z_1^2+z_2^2+z_3^2-2z_1z_2-2z_2z_3-2z_3z_1)}{1+\frac r5z_1z_2
z_3(z_1^2z_2^2+z_2^2z_3^2+z_3^2z_1^2-2z_1z_2z_3^2-2z_1z_2^2z_3-2z_1^2z_2z_3)},$$
is not Schur--Agler for $5/6<r<1$. For these values of $r$, the stable polynomial
$$q(z_1, z_2, z_3)=1+\frac r5\ z_1z_2z_3\bigg(z_1^2z_2^2+z_2^2z_3^2+z_3^2z_1^2-2z_1z_2z_3^2-2z_1z_2^2z_3-2z_1^2z_2z_3\bigg)$$
is not an Agler denominator.
\end{ex}

We now have the following result.

\begin{thm}\label{Aglerden} Let a polynomial $p$ admit a representation
\eqref{eq:repr} for some $n\in\mathbb N_0^d$ and contractive $K$.
Then
\begin{equation}\label{prevoverp}
\frac{z^n\bar{p}(1/z)}{p(z)} =\det (-K^* + \sqrt{I-K^*K} Z_n (I-KZ_n)^{-1} \sqrt{I-KK^*} ).
\end{equation}
In particular, $p$ is an eventual Agler denominator of order $n$.
If $\deg p=n$, then $p$ is an Agler denominator.
\end{thm}

A lemma is needed; see, e.g., \cite[Theorem 3.1.2]{Prasolov}.
\begin{lem}\label{commutingAC} Let $A,B,C$, and $D$ be square matrices of the same size,
and suppose that $AC=CA$. Then $$ \det \begin{bmatrix} A & B\\
 C &
D
\end{bmatrix} = \det (AD-CB). $$
\end{lem}

\begin{proof}[Proof of Theorem \ref{Aglerden}] By Lemma
\ref{lem:Schur-compl}, the right hand side of \eqref{prevoverp}
equals
\begin{equation}\label{eq1}
\frac{\det\begin{bmatrix} -K^* & -\sqrt{I-K^*K}Z_n
\cr \sqrt{I-KK^*} & I-KZ_n \end{bmatrix}}{\det (I-KZ_n)}.
\end{equation} Let $K=U\Sigma V^*$ be a
singular value decomposition of $K$. Then the numerator of \eqref{eq1} equals
\begin{equation}\label{eq1a}
\det\begin{bmatrix} VU^* & 0 \cr 0 & I \end{bmatrix}\det\begin{bmatrix} -U\Sigma U^* & -U\sqrt{I-\Sigma^2}V^*Z_n
\cr U\sqrt{I-\Sigma^2}U^* & I-KZ_n \end{bmatrix}.
\end{equation}
Applying Lemma \ref{commutingAC}, noting that
$-U\Sigma U^*$ and $U\sqrt{I-\Sigma^2} U^*$ commute, we get that \eqref{eq1a} equals
\begin{equation*}\label{eq2} \det (VU^*) \det (-U \Sigma U^* (I-U\Sigma V^*Z_n) + U\sqrt{I-\Sigma^2} U^* U \sqrt{I-\Sigma^2} V^*Z_n) \end{equation*} \begin{equation}= \det (Z_n-K^*).
\end{equation}
To prove \eqref{prevoverp} it remains to observe that
\begin{equation*}\label{prev}z^n \overline{p}(1/z) = z^n \det(I-\overline{K}Z_n^{-1}) = \det(I-\overline{K}Z_n^{-1})\det Z_n\end{equation*} \begin{equation} = \det (Z_n-\overline{K}) = \det(Z_n-K^*), \end{equation}
where in the last step we used that $Z_n^\top=Z_n$. As the Julia
operator
$$ \begin{bmatrix} -K^* & \sqrt{I-K^*K} \cr \sqrt{I-KK^*} & K \end{bmatrix}   $$
%$$ \begin{bmatrix} U & 0 \cr 0 & U \end{bmatrix} \begin{bmatrix} -\Sigma & \sqrt{I-\Sigma^2} %\cr \sqrt{I-\Sigma^2} & \Sigma \end{bmatrix} \begin{bmatrix} U^* & 0 \cr 0 & V^* %\end{bmatrix}$$
is unitary, the multivariable rational inner matrix function
$$ -K^* + \sqrt{I-K^*K} Z_n (I-KZ_n)^{-1} \sqrt{I-KK^*}  $$ is in the Schur--Agler class.
By Corollary \ref{detcor}, so is its determinant, and thus
$z^n\bar p(1/z)/p(z)$ is in the Schur--Agler class.
\end{proof}
\begin{cor}\label{cor:shrink}
For every $p\in\mathbb{C}[z_1,\ldots,z_d]$ with $p(0)=1$, there
exists $r>0$ such that the polynomial $p_r(z):=p(rz)$ is an
eventual Agler denominator. In fact, if $p$ is given by \eqref{eq:repr}, then one can choose any
$0<r\le1/\|K\|$.
\end{cor}
\begin{proof}
Since, by Theorem \ref{thm:k-repr}, every polynomial $p$ with
$p(0)=1$ admits a representation \eqref{eq:repr}, the assertion
follows from the identity $$p_r(z)=\det(I_{|n|}-rKZ_n),$$ and
Theorem \ref{Aglerden}.
\end{proof}
\begin{rem}\label{rem:shrink}
The first statement of Corollary \ref{cor:shrink} can also be
deduced from Corollary \ref{cor:q-scalar}: since
$\|1-p(0)\|_{\mathcal{A}}=0$, the inequality
$\|1-p_r\|_\mathcal{A}\le1$ holds, by continuity, for a
sufficiently small $r>0$. For multi-affine symmetric polynomials a
stronger statement is true \cite[Theorem 1.5]{KneseSymm}: $p_r$ is
an Agler denominator for sufficiently small $r>0$.
\end{rem}

Following \cite{BF}, we call a semi-stable polynomial $p$ {\it scattering
Schur} if $p$ and $z^{\deg p}\bar{p}(1/z)$ have no factor in common.
In \cite[Theorem 1]{Kummert} it was proven that every two-variable scattering Schur polynomial $p$
of degree $n=(n_1,n_2)$ is of
the form \eqref{eq:repr} with $K$ an $(n_1+n_2) \times (n_1+ n_2)$ contraction.
Thus every two-variable scattering Schur polynomial $p$ is an Agler denominator.

The following result provides a partial converse to Theorem \ref{Aglerden}.

\begin{thm}\label{Aglerdenrev} Let $p$ be a $d$-variable scattering Schur polynomial with $p(0)=1$.
If, for some $m\in {\mathbb N}_0^d$, the rational inner function $z^m \bar{p}(1/z)/p(z)$ has a
 transfer-function realization \eqref{eq:tf} of
order $m$, then $p$ admits a representation \eqref{eq:repr} with
$n=m$ and $K$ a contraction.
\end{thm}

\begin{proof} Taking the determinant of both sides of the equality
$$ \frac{z^m \bar{p}(1/z)}{p(z)} = A + BZ_m (I-DZ_m)^{-1}C$$
and using Lemma \ref{lem:Schur-compl}, we obtain
$$ \frac{z^m \bar{p}(1/z)}{p(z)} = \frac{\det \begin{bmatrix} A & -BZ_m \cr C & I-DZ_m
\end{bmatrix}}{\det(I-DZ_m)} =: \frac{r(z)}{s(z)}. $$
Note that both $r(z)$ and $s(z)$ are of degree at most $m$. We now
obtain that
$$ {(z^m \bar{p}(1/z))}{s(z)} = r(z) p(z) . $$
As $p$ is scattering Schur we must have that $p(z)$ divides
$s(z)$, say $s(z) = q(z) p(z)$. Dividing out $p(z)$ in the above
equation, we obtain that
$${(z^m \bar{p}(1/z))}{q(z)} = r(z) . $$
As the left hand side has degree $m+ \deg q$ and the right hand
side degree at most $m$, we obtain that $q$ must be a constant.
But then, using $p(0)=1$ and $s(0)= \det (I-DZ_m)|_{z=0} = 1$, we
obtain that $q=1$, and thus $p(z) = \det (I-KZ_m)$ with $K=D$.
\end{proof}

\begin{cor}\label{affine}
The polynomials $p(z_1, \ldots , z_d) =  1 - \sum_{i=1}^d a_i z_i$
with $\sum_{i=1}^d |a_i| \le 1$, are Agler denominators.
\end{cor}

\begin{proof} Let $K$ be a $d\times d$ rank 1 contraction with diagonal entries $a_1, \ldots ,
a_d$. One such choice is given by $$K=\bigg[\sqrt{|a_j a_k|}e^{i\
{\rm arg}a_k}\bigg]_{j,k=1}^d. $$ Then $\det
(I_d-KZ_{(1,\ldots,1)})=p(z)$ and the result follows directly from
Theorem \ref{Aglerden}.
\end{proof}

\begin{rem}\label{rem:min}
The matrix $K$ in the proof of Corollary \ref{affine} is clearly
of minimal size. It is also of minimal norm,
$\|K\|=|a_1|+\cdots+|a_d|$, for otherwise $p(z)=\det
(I_d-KZ_{(1,\ldots,1)})$ would be stable.
\end{rem}

  It was shown in \cite[Theorem 3.3]{KneseSymm}
that a multi-affine symmetric polynomial is an Agler denominator
if and only if a certain matrix ${\mathcal B}$ constructed from
the Christoffel--Darboux equation is positive semidefinite.
Subsequently, for $p(z) =1 - \frac{1}{d}\sum_{i=1}^d z_i$, the
positivity of the matrix $\mathcal{B}$ was
 computationally checked up to $d=11$. Using our Corollary
\ref{affine}, we deduce this fact for all $d$.

\begin{cor} Let $p(z) = t -
\frac{1}{d}\sum_{i=1}^dz_i$, where $|t|\ge 1$, and let
$$ {\mathcal B} := (B_{|\alpha |, |\beta |}^{|\alpha \cap
\beta|})_{\alpha, \beta \subseteq \{ 1,\ldots , d-1 \} } , $$ be
the $2^{(d-1)\times (d-1)}$ matrix indexed by subsets $\alpha ,
\beta $ of $\{ 1,\ldots , d-1 \}$, defined via:
\begin{equation}
{d \choose j}^{-1} {d \choose k}^{-1} (p_j\bar p_k - \bar p_{d-j}
p_{d-k}) = (d-j-k+i)B_{j,k}^i-iB_{j-1,k-1}^{i-1},   \end{equation} where $0 \le i \le j,k
\le d-1, p_0=t$, $p_1=\frac{1}{d}$,
$p_j=0$, $j\ge 2$, and $B_{j,k}^{i} = 0$ for $i,j,k$ not
satisfying $ 0 \le i \le j,k \le d-1$. Then ${\mc B}$ is
positive semidefinite.
\end{cor}
To illustrate, we choose  $t$ real and $d=3$:
$$ {\mc B} = \frac{1}{18} \begin{bmatrix} 6t^2 & -3t & -3t & 0 \cr -3t & 3t^2+1 & 2 & -3t\cr -3t & 2 & 3t^2+1 & -3t \cr 0 & -3t & -3t & 6t^2 \end{bmatrix}=A^*A, $$
where
$$ A = \frac{1}{\sqrt{18}}\begin{bmatrix} \sqrt{6}t & -\frac{1}{2}\sqrt{6} &  -\frac{1}{2}\sqrt{6} & 0 \cr 0 & -\frac{1}{2}\sqrt{6} &  -\frac{1}{2}\sqrt{6} &  \sqrt{6}t \cr 0 & \sqrt{3(t^2-1)}  & \sqrt{3(t^2-1)} & 0 \cr 0 & 1 & -1 & 0  \end{bmatrix} . $$

\begin{rem}\label{rem:linear}
In the context of \eqref{eq:repr}, the question of whether a given
polynomial is an (eventual) Agler denominator is reduced to
analyzing the matrix of its determinantal representation. This
provides a possible alternative to the transfer-function
realization method. For example, $p(z)=1-\frac{1}{3}(z_1+z_2+z_3)$
admits a representation \eqref{eq:repr} with
$$K=\frac{1}{3}\begin{bmatrix}
1 & 1 & 1 \\
1 & 1 & 1 \\
1 & 1 & 1
\end{bmatrix},$$
which is minimal both in size and in norm; see Remark
\ref{rem:min}. At the same time, the minimal order of a
transfer-function realization \eqref{eq:tf} of
$$z_1z_2z_3\frac{\bar{p}(1/z)}{p(z)}=\frac{3z_1z_2z_3-z_2z_3-z_1z_3-z_1z_2}{3-z_1-z_2-z_3}$$ is $m=(2,2,2)$
\cite{Knese2011,BK}.
\end{rem}

\section{Variations on the Kaijser--Varopoulous--Holbrook example}\label{sec6}

For $s$ real, consider the multivariable polynomial
$$\ds p(z_1, \ldots, z_d)=(1+s)\sum_{m=1}^dz_m^2-\bigg(\sum_{m=1}^dz_m\bigg)^2. $$
The case of $d=3$ and $s=1$ corresponds to the polynomial from \cite{Varo, Hol01}.

\begin{prop}\label{sharper} Let $d>1$, $s>d/2 - 1$, and $p$ be defined as above. Then
$\| p \|_\infty=(1+s)d$, if $d$ is even, and $\| p \|_\infty<(1+s)d$, if $d$ is odd,
while $\|p\|_{\mc A} =(1+s)d$ for all $d$. In particular, $p/\| p \|_\infty$ is not in $\mc{SA}_d$ for odd $d>1$.
\end{prop}
\begin{proof} Write $p(z_1, \ldots, z_d) = z^\top Az$, where
$$ A=
\begin{bmatrix}
s & -1 & \ldots &-1\\
-1 & s & \ldots &-1\\
  & & \ddots & \\
-1 & -1 &\ldots & s
\end{bmatrix}\quad {\rm and}\quad z= \begin{bmatrix} z_1 \\ \vdots \\ z_d \end{bmatrix}.$$
 Observe that $A$ is symmetric with eigenvalues $1+s$ and $s-d+1$, so  $$\ds\|A\|=\max\big\{1+s,\ |s-d+1|\big\} = 1+s.$$
Hence $\| p \|_\infty \le(1+s)d$. If $d$ is even, one immediately has equality since $p(1, -1, \ldots, 1, -1)=(1+s)d$.
If $d$ is odd, the inequality is strict. Indeed, otherwise $|p|$ would be maximized for some unimodular $z_1,\ldots, z_d$ with zero sum,
as $z$ would then lie in the eigenspace of $A$ corresponding to $1+s$.
But then the equality $$|p(z_1,\ldots, z_d)|=(1+s) \left|\sum_{i=1}^d z_i^2\right|=(1+s)d$$
would force $z_i=\pm e^{i\alpha},\ i=1,\dots, d,$ in conflict with the zero sum condition.
%$z_i$, odd in number, from summing to zero.
%Note that the eigenspace corresponding to $1+s$ is $\{z\colon z_1+\ldots+z_d = 0 \}$.

Next, for a tuple of commuting contractions $T=(T_1 , \ldots , T_d )$, we have
$$p(T_1, \dots, T_d) =\begin{bmatrix}
T_1 & \ldots & T_d \end{bmatrix} (A\otimes I) \begin{bmatrix}
T_1\\
\vdots\\
T_d \end{bmatrix},$$ so that
$$\| p(T_1, \dots, T_d)\|  \le \| A\|d=(1+s)d.$$ Choose $v_1,
\ldots, v_d$ to be any unit vectors in $\mathbb R^{2}$ with zero
sum. Then the matrices
$$\ T_i=\begin{bmatrix}
0 & v_i^\top & 0 \\
0 &  0    & v_i\\
0 &  0    & 0 \end{bmatrix} \in {\mathbb R}^{4\times 4},\qquad
i=1,\ldots, d,$$ are such that $\ \|T_i\|=1,\ $
$T_iT_j=T_jT_i=\langle v_i, v_j\rangle e_1 e_{4}^\top ,\ $
$\sum_{i=1}^d T_i=0$, and
$$ p(T_1,\dots,T_d)e_{4}=(1+s)de_{1}, $$
where $e_{j}$ is the $j$th standard unit vector in $\mathbb
R^{4}$. Hence $ \| p \|_{\mathcal A} =(1+s)d$.
\end{proof}
\begin{rem}
In the case of $d=3$, maximizing  $\| p \|_{\mathcal A} / \| p \|_\infty$, the von Neumann constant of $p$, over $s$, we find that the maximum possible ratio is $\frac13\sqrt{\frac{35+13\sqrt{13}}6}\approx1.23$ (occuring for $s=\frac{\sqrt{13}+1}6$).
The previously known lower bound for the von Neumann constant was $\frac{6}{5}$ \cite{Hol01}.
\end{rem}

\subsection*{Acknowledgment}
We thank Victor Vinnikov and Bernd Sturmfels for energizing
discussions, Greg Knese and the anonymous referee for valuable
suggestions, and
 David Scheinker for bringing
 \cite{Borcea,Branden} to our attention.

\end{document}